\documentclass[11pt,letterpaper]{amsart}
\usepackage{amssymb, amscd, amsfonts, euler, epic, booktabs}
\usepackage[all]{xy}
\usepackage{srcltx}

\usepackage{psfrag}
\usepackage{graphicx}
\usepackage{epstopdf}

\xyoption{all}
\usepackage{charter}

\makeatletter
\newenvironment{tablehere}
  {\def\@captype{table}}
  {}

\makeatletter
\newenvironment{figurehere}
  {\def\@captype{figure}}
  {}
\makeatother

\def\AA{{\mathbb A}}

\def\CC{{\mathbb C}}

\def\GG{{\mathbb G}}

\def\PP{{\mathbb P}}
\def\QQ{{\mathbb Q}} 
\def\RR{{\mathbb R}}

\def\ZZ{{\mathbb Z}} 
\def\Ltil{{\tilde{L}}}

\def\Sfrak{{\mathfrak S}}

\def\R{R} 
\def\Rtil{{\tilde{\R}}} 
\def\Rhat{{\hat{\R}}}

\def\root{{\mathfrak{R}}}
\def\T{{T}} 
\def\Tcal{{\mathcal{\T}}} 
\def\Aff{{\mathbb{A}}}
\def\two{{\beta}}
\def\simpl{{K}}
\def\we{\lambda}
\def\subgr{\gamma}
\def\sect{\sigma}
\def\refl{{s}}

\def\G{{\Gamma}}

\def\hyp{{\rm hyp}} 
\def\odd{{\rm odd}} 
\def\even{{\rm even}}

\def\reg{{\rm reg}}

\def\Sym{\mathrm{Sym}}
\def\bs{\backslash}

\def\eps{\epsilon}

\def\Ccal{{\mathcal C}}
\def\Ccan{\breve{C}}

\def\Dcal{{\mathcal D}}
\def\Ecal{{\mathcal E}} 
\def\Ecalhat{{\hat{\Ecal}}}
\def\Fcal{{\mathcal F}} 
 
\def\Hcal{{\mathcal H}} 
 
\def\Jcal{{\mathcal J}} 

\def\Lcal{{\mathcal L}}
\def\Lcalhat{{\hat{\Lcal}}}
\def\Lhat{{\hat{L}}}
\def\Lcaltil{\tilde{\Lcal}}

\def\Mcal{{\mathcal M}}

\def\Ocal{{\mathcal O}}
\def\Ofrak{{\mathfrak{O}}}
\def\Pcal{{\mathcal P}}
\def\Scal{{\mathcal S}}
  
\def\Vcal{{\mathcal V}}

\def\Ycal{{\mathcal{Y}}}
\def\Yhat{{\hat{Y}}}
\def\Ycalhat{{\hat{\Ycal}}}
\def\Sym{\mathrm{Sym}}

\def\root{{\mathfrak{R}}}
\def\perm{{\mathfrak{S}}}

\def\la{\langle}
\def\ra{\rangle}

\def\rar{\rightarrow}
\def\can{\omega}

\def\Kcl{K}
\def\Rcal{{\mathcal{R}}}

\def\ab{\varphi}
\def\Ab{\Phi}

\def\kbold{{\mathbf k}}

\newcommand\art{\operatorname{Art}}
\newcommand\aut{\operatorname{Aut}}

\newcommand\Hom{\operatorname{Hom}}

\newcommand\Ker{\operatorname{Ker}}

\newcommand\pic{\operatorname{Pic}}

\newcommand\GL{\operatorname{GL}}

\newcommand\SL{\operatorname{SL}}

\newcommand\supp{\operatorname{supp}}

\newtheorem{theorem}{Theorem}[section]
\newtheorem{lemma}[theorem]{Lemma}
\newtheorem{proposition}[theorem]{Proposition}

\newtheorem{corollary}[theorem]{Corollary}

\theoremstyle{definition}

\theoremstyle{remark}

\title{The fine structure of Kontsevich-Zorich strata for genus 3}

\author{Eduard Looijenga}
\address{Mathematisch Instituut\\
Universiteit Utrecht\\ P.O.~Box 80.010, NL-3508 TA Utrecht
(Nederland)}
\email{E.J.N.Looijenga@uu.nl}
\author{Gabriele Mondello}
\address{``Sapienza'' Universit\`a di Roma\\  
Dipartimento di Mathematica ``Guido Castelnuovo''\\
Piazzale Aldo Moro 5,  00185 Roma (Italia)}
\email{mondello@mat.uniroma1.it}

\begin{document}
\maketitle

\begin{abstract}
We give a description of the Kontsevich-Zorich strata for genus 3 in terms of root system data. For each non-open stratum we obtain  a presentation of its orbifold fundamental group. 
\end{abstract}

\section*{Introduction}

The moduli space of pairs $(C,\ab)$ with $C$ a complex connected smooth projective curve of genus $g\ge 2$ and $\ab$  a nonzero abelian differential on $C$, and denoted here by  
$\Hcal_g$, comes with the structure of 
a Deligne-Mumford stack, but we will just regard it as an orbifold. The forgetful morphism $\Hcal_g\to\Mcal_g$ exhibits $\Hcal_g$ as the complement of the zero section of the Hodge bundle over $\Mcal_g$. A partition of $\Hcal_g$ into suborbifolds is defined by looking at  the multiplicities of the zeros of  the abelian differential: for  any numerical partition $\kbold:=(k_1\ge k_2\ge \cdots \ge k_n>0)$ of $2g-2$, we have a subvariety  $\Hcal'(\kbold)\subset \Hcal_g$ which parameterizes the pairs $(C,\ab)$ for which the zero divisor $Z_\ab$ of $\ab$  is of type $\kbold$. We call  a \emph{stratum} of $\Hcal_g$  a connected component of some $\Hcal '(\kbold)$ (NB: our terminology  differs from that of \cite{kz:2003}). 
 
\subsection*{Classification of strata}
Kontsevich and Zorich  \cite{kz:2003} characterized the strata of $\Hcal_g$ in rather simple terms.  First consider the case when $C$ is 
hyperelliptic. Then an effective divisor of degree $2g-2$ on $C$ is canonical if and only if it is invariant under the hyperelliptic involution. So the type $\kbold$ of such a divisor has the property that any odd integer appears in it an even number of times.  There are two cases where the support of the canonical divisor is an orbit of the hyperelliptic involution: one is of type $(2g-2)$ and the other is of type $(g-1,g-1)$ (the two cases corresponding to a Weierstra\ss\ point resp.\  a pair of points). The authors show  that these  make up strata and denote them by 
$\Hcal^{\hyp}(2g-2)$  and  $\Hcal^{\hyp}(g-1,g-1)$ respectively. Notice that this fully covers the case $g=2$. They show that for $g\ge 3$, 
\[
\Hcal(\kbold):=
\begin{cases}
\Hcal'(\kbold)-\Hcal^\hyp (\kbold) &\text{for $\kbold=(2g-2), (g-1,g-1)$},\\
\Hcal'(\kbold) & \text{otherwise,}
\end{cases}
\]
is a stratum unless
$g>3$ and all the terms of $\kbold$ are even. In that case the canonical divisor is twice the divisor of an (effective) theta characteristic, which for $g>3$ can be even or odd (for $g=3$ it is necessarily odd, as there is no effective even theta characteristic).  These loci are connected  and hence define strata: $\Hcal^\even (\kbold)$ and $\Hcal^\odd(\kbold)$. This completes the Kontsevich-Zorich characterization of the strata. 

\subsection*{Local structure of strata}
Each stratum is known to have a `linear' structure: it comes with an atlas of holomorphic  charts whose transition maps lie in $\GL (d,\CC)$, where $d$ is the dimension of the stratum. A chart of this atlas at $(C,\ab)$ is defined as follows.  Note that $\ab$ defines an element  
$[\ab]\in H^1(C, Z_\ab;\CC)$, where $Z_\ab$ denotes the zero locus consisting of $n$ distinct points.  If we put $d:=2g+|Z_\ab|-1=\dim H^1(C, Z_\ab;\CC)$ and choose an isomorphism $H^1(C, Z_\ab;\CC)\cong\CC^d$, then varying  $(C,\ab)$ in a small  neighborhood in its stratum makes the image of  $[\ab]$ vary in $\CC^d$. This yields a $\CC^d$-valued chart.  In particular $\dim\Hcal (\kbold)=2g+n-1$. 

\subsection*{Projective classes of abelian differentials}
Since every connected stratum $\Scal$ of $\Hcal_g$ is invariant under scalar multiplication it defines a suborbifold $\PP \Scal$ in the $\PP^{g-1}$-orbifold bundle $\PP\Hcal_g$ over $\Mcal_g$. Such a projectivized stratum parameterizes pairs $(C,D)$  with $C$ a smooth projective curve and $D$ a positive canonical divisor on $C$ with prescribed multiplicities.
Moreover, $\Scal$ has a contractible
orbifold universal cover if and only if $\PP \Scal$ has: in this case, the orbifold fundamental group of $\PP \Scal$  is the quotient of the orbifold fundamental group of $\Scal$ by an infinite cyclic central subgroup.
The somewhat technical price to pay for working with $\PP \Scal$ is that over the hyperelliptic locus we have to deal with $\ZZ/2$-gerbes.

Since $\Scal$ has a linear structure, $\PP\Scal$ has a projective structure (i.e., it has a holomorphic  atlas which takes values in $\PP^{\dim \Scal -1}$ such that the transition maps lie in a projective linear group).

\subsection*{Hyperelliptic strata}
The topology of the hyperelliptic strata is familiar and although a bit subtle, essentially well understood.
Consider an affine plane curve $C_\rho^\circ$ with equation
\[
w^2=\prod_{i=1}^{2g+1} (z-\rho_i)
\]
where the $\rho_i$'s are
pairwise distinct. A projective curve $C_\rho$ is obtained by adding a point $p$ at infinity and $w^{-1}dz$ extends over $C_\rho$ as an abelian differential $\ab_\rho$ with a unique zero at infinity.
Notice that $\ab_\rho$ naturally defines a nonzero element in $(T^*_{p}C_\rho)^{\otimes (2g-1)}$.

Performing the above construction in families over
\[
\Rcal_{2g+1}:=\{\rho\in\CC^{2g+1}\,|\,\sum_i \rho_i=0,\ \rho_i\neq \rho_j\}
\]
we obtain a diagram 
\[
\xymatrix{
\Ccal \ar[d]^f & \Ccal^\circ\ar@{_(->}[l]\ar@{^(->}[r] & \Rcal_{2g+1}\times \AA^2 \ar[lld] \\
\Rcal_{2g+1} \ar@/^1pc/[u]^\sigma
}
\]
where $f$ is a smooth projective curve with a section $\sigma$.
Moreover, we also obtain a $\Ab\in H^0(\Ccal,\can_f)$ that
vanishes along the image of $\sigma$ of order $2g-2$
and
so a nowhere-zero section of $(\sigma^*\can_f)^{\otimes(2g-1)}$ over $\Rcal_{2g+1}$.

The $\GG_m$-action on $\Rcal_{2g+1}$ defined as
$\zeta\cdot\rho:= \zeta^2\rho$
lifts on $\Rcal_{2g+1}\times\AA^2$ as
$\zeta\cdot(\rho,z,w):=(\zeta^2\rho,\zeta^2 z,\zeta^{2g+1}w)$
preserving $\Ccal^\circ$ and $\sigma$. Notice that $-1\in\GG_m$ yields the
hyperelliptic involution.
The induced action on $\Ab$ is
$\zeta\cdot\ab_\rho:=\zeta^{1-2g}\ab_{\zeta^2\rho}$,
and so $(C_\rho,\ab_\rho)$ is isomorphic to
$(C_{\rho'},\ab_{\rho'})$ if and only if $\rho'=\zeta\cdot\rho$
with $\zeta\in\mu_{2g-1}$.
Moreover, the action of
$\perm_{2g+1}$ on $\Rcal_{2g+1}$ that permutes
the components of $\rho$ commutes with the $\GG_m$-action
and so
we obtain an isomorphism of orbifolds
\[
\xymatrix{
\Rcal_{2g+1}/(\perm_{2g+1}\times\mu_{2g-1}) \ar[rr]^{\qquad\sim}
&&
\Hcal^{hyp}(2g-2)
}
\]
Now, $\Rcal_{2g+1}/\perm_{2g+1}$
is a classifying space for the braid group $B_{2g+1}$
on $2g+1$ strands. Hence $\Hcal^{hyp}(2g-2)$ is an orbifold
classifying space for a group that is an  
extension of $\mu_{2g-1}$ by $B_{2g+1}$.
If we are interested in projective classes of
abelian differentials, then we obtain
\[
\xymatrix{
\Rcal_{2g+1}/(\perm_{2g+1}\times\GG_m)
\cong \Mcal_{0,2g+2}/\perm_{2g+1}
\ar[rr]^{\qquad\qquad\sim} && \PP\Hcal^{hyp}(2g-2)
}
\]
and so $\PP\Hcal^{hyp}(2g-2)$ has a contractible orbifold universal cover and its orbifold
fundamental group is a hyperelliptic mapping class group, namely 
the centralizer in the mapping class group of a hyperelliptic involution $\tau$ of a pointed genus $g$ 
surface  (which  preserves the  point) and that the orbifold universal cover is contractible.

Similarly for $\Hcal^{hyp}(g-1,g-1)$, we consider affine plane curves $C^\circ_\rho$ of equation
\[
w^2=\prod_{i=1}^{2g+2}(z-\rho_i)
\]
together with the differential $\ab_\rho=w^{-1}dz$.
If $C_\rho$ is the smooth completion of $C^\circ_\rho$
obtained by adding the points at infinity $p_1$ and $p_2$,
then $\ab_\rho$ naturally determines a nonzero element in
$(T^*_{p_1}C_\rho \otimes T^*_{p_2}C_\rho)^{\otimes g}$.

Implementing the above construction in families yields
a smooth curve $f:\Ccal\rar\Rcal_{2g+2}=\{\rho\in\CC^{2g+2}\,|\,\sum_i\rho_i=0,\,\ \rho_i\neq\rho_j\}$ together with
a divisor $\Dcal$ that projects \'etale $2:1$ onto $\Rcal_{2g+2}$,
a $\Phi\in H^0(\Ccal,\can_f)$ that vanishes of order $g-1$ along $\Dcal$ and a nowhere-zero section of
$\det(f_*(\can_f\otimes\Ocal_{\Dcal}))^{\otimes g}$.
A $\GG_m$-action can be defined
on $\Rcal_{2g+2}\times\AA^2$ as
$\zeta\cdot(\rho,z,w):=(\zeta\rho,\zeta z,\zeta^{g+1}w)$
and the induced action on $\Ab$ is
$\zeta\cdot\ab_\rho:=\zeta^{-g}\ab_{\zeta^2\rho}$.
We also have the obvious action of $\perm_{2g+2}$ in this family
(which just permutes the $\rho_i$'s) and the involution 
$\tau$ which sends $(\rho,z,w)$ to $(\rho,z,-w)$. These three actions commute, so that
we have one of $\perm_{2g+2}\times \perm_2\times \GG_m$. The stabilizer
of $\Phi$ is $\perm_{2g+2}\times \perm_2\times \mu_g$ and 
it is easy to see that any isomorphism $(C_\rho,\ab_\rho)\cong (C_{\rho'},\ab_{\rho'})$ 
is the restriction of an element of this stabilizer. So we obtain the isomorphism
of orbifolds
\[
\xymatrix{
\Rcal_{2g+2}/(\perm_{2g+2}\times\perm_2\times \mu_{g}) \ar[rr]^{\quad\sim}
&&
\Hcal^{hyp}(g-1,g-1)
}
\]
This makes $\Hcal^{hyp}(g-1,g-1)$ an orbifold classifying space
of an extension of $\perm_2\times \mu_{g}$
by $B_{2g+2}$.
It also shows that the orbifold fundamental group of $\PP\Hcal^{hyp}(g-1,g-1)$ is just a hyperelliptic mapping class group
(to be precise, it is the centralizer in the full mapping class group of a hyperelliptic involution of a twice pointed genus $g$ 
surface  which exchanges the points) and that the orbifold universal cover is contractible.


In their preprint \cite{kz:97}
Kontsevich and Zorich conjecture that something similar is true in general, namely that each projectivized stratum always has a contractible orbifold universal cover and that its orbifold fundamental group is commensurable with
some mapping class group. 

\subsection*{The other strata in genus 3}
In this paper we give rather precise descriptions of all the strata in genus 3. This enables us to find a presentation of their
orbifold fundamental group, at least in principle: we do this for all the strata, except  for the open stratum $\Hcal(1^4)$, where it gets unwieldy, and to make for these strata the Kontsevich-Zorich conjecture so explicit that it acquires more of  a  topological flavor. 
Concretely, we show that the nonhyperelliptic strata in genus 3 can be understood as parameterizing del Pezzo surfaces of degree two or one endowed with an anticanonical divisor of a given type and describe these in turn in terms of combinatorial (root) data. Such moduli spaces have been investigated by one of us before. It turns out  that they can be given in  a somewhat similar spirit as the hyperelliptic strata: some  orbifold cover appears as the complement of a locally finite arrangement in a domain and the in principle their fundamental group can be computed. For instance, for $\Hcal (3,1)$ resp.\ $\Hcal (4)$ we get the discriminant complement of the root system of type $E_7$ resp.\ $E_6$. Its fundamental group is the Artin group of that type and  a highly nontrivial theorem of Deligne \cite{deligne} asserts that this  complement has indeed a contractible universal cover. But in the other cases this seems difficult to establish.  Questions of that kind are reminiscent, and indeed overlap,
with the Arnol'd-Thom conjecture which states that the discriminant complement of the universal deformation of an an isolated hypersurface singularity is a $K(\Gamma, 1)$. They are still the subject of current research \cite{allcock}. 

We do not know how to make the commensurability conjecture of Kontsevich and Zorich precise. Our results seem to indicate that the open stratum $\Hcal (1^4)$ has an orbifold fundamental group which  may \emph{not} be commensurable with a central extension of the mapping class group of the punctured genus 3 curve.  In fact, for none of the strata described here, we were able to characterize  their orbifold fundamental group as a kind of a mapping class group.

\subsection*{Acknowledgements}
E.L.\ wishes to thank the Mathematical Sciences Research Institute and the Tsinghua Mathematics Department for support and hospitality during the period part of this work was done. G.M.\ would like to thank the Park City Mathematical Institute for the hospitality in July 2011.

\section{Genus 3 strata in terms of del Pezzo surfaces}\label{sect:genus3}

Let $C$ be a nonhyperelliptic nonsingular curve of genus $3$. The canonical system on $C$ embeds in a projective plane $P_C$.
A double cover $\pi: X_C\to P_C$ of this plane which totally ramifies over $C$ is unique up to the obvious involution. The covering variety $X_C$ is a del Pezzo surface of degree 2, which here amounts to saying that the anticanonical system on $X_C$ is ample and is given by the morphism  $\pi$: any effective anticanonical divisor on $X_C$ is the pull-back of a line in $P_C$.

\subsection*{Brief review of degree two del Pezzo surfaces}
Every effective anticanonical divisor on $X_C$ has arithmetic genus one
(this is also clear from the  way we  defined $X_C$). If $L\subset P_C$ is a double tangent
of $C$, then its preimage in  $X_C$ consists of two exceptional curves (an \emph{ exceptional curve} is a smooth genus zero curve with self-intersection $-1$) which intersect each other with intersection number $2$. Since there are 28 bitangents we get $2\cdot 28=56$ exceptional curves and these are in fact all of them. If we select $7$ such exceptional curves that are pairwise disjoint, then their contraction yields a copy of $\PP^2$ and the anticanonical system is then realized as the system of cubics passing through the 7 image points of this $\PP^2$. There are however many ways of picking 7 pairwise disjoint copies and the best way to come to terms with this is to invoke an associated symmetry group, which is a Weyl group of type $E_7$. 

Let us make this precise. The natural map $\pic (X_C)\to H^2(X_C)$ is an isomorphism. The latter is free of rank 8 and its intersection pairing can be diagonalized with on the diagonal a $1$ followed by seven $-1$'s. If $\Kcl\in \pic (X_C)$ stands for the anticanonical class, then it  is clear that $\Kcl\cdot \Kcl=2$. So the  orthogonal complement of $\Kcl$ in $\pic (X_C)$, denoted here by $\pic_\circ(X_C)$,  is negative definite.
A \emph{root} is an element $\alpha\in \pic_\circ(X_C)$
with $\alpha\cdot\alpha =-2$. Associated to a root $\alpha$ is a reflection $s_\alpha$ in $\pic (X_C)$, given by $u\in \pic (X_C)\mapsto u+(u\cdot\alpha)\alpha$, which  fixes $\Kcl$ and preserves the intersection pairing. These reflections generate a Weyl group $W(X_C)$ of type $E_7$ so that  the roots make up a root system $\root(X_C)$ of the same type.  

It is known that the classes of the exceptional curves generate $\pic (X_C)$ and that the roots generate $\pic_\circ(X_C)$. A root can be represented by the difference of two disjoint exceptional curves, although not uniquely so. The 
Weyl group $W(X_C)$ is the full stabilizer of $\Kcl$ in the orthogonal group of $\pic (X_C)$. The involution of $X_C$ preserves $\Kcl$ and acts as minus the identity in $\pic_\circ(X_C)$, hence appears here as the central element of $W(X_C)$.  But the other nontrivial elements  of
$W(X_C)$ are usually not induced by an automorphism of $X_C$.

\begin{center}
\begin{figurehere}
\psfrag{L}{$L_D$}
\psfrag{K}{$\Ltil_D$}
\psfrag{P}{$P_C$}
\psfrag{X}{$X_C$}
\psfrag{C}{$C$}
\psfrag{(4)}{$(4)$}
\psfrag{(3,1)}{$(3,1)$}
\psfrag{(2^2)}{$(2^2)$}
\psfrag{(2,1^2)}{$(2,1^2)$}
\psfrag{(1^4)}{$(1^4)$}
\includegraphics[width=0.82\textwidth]{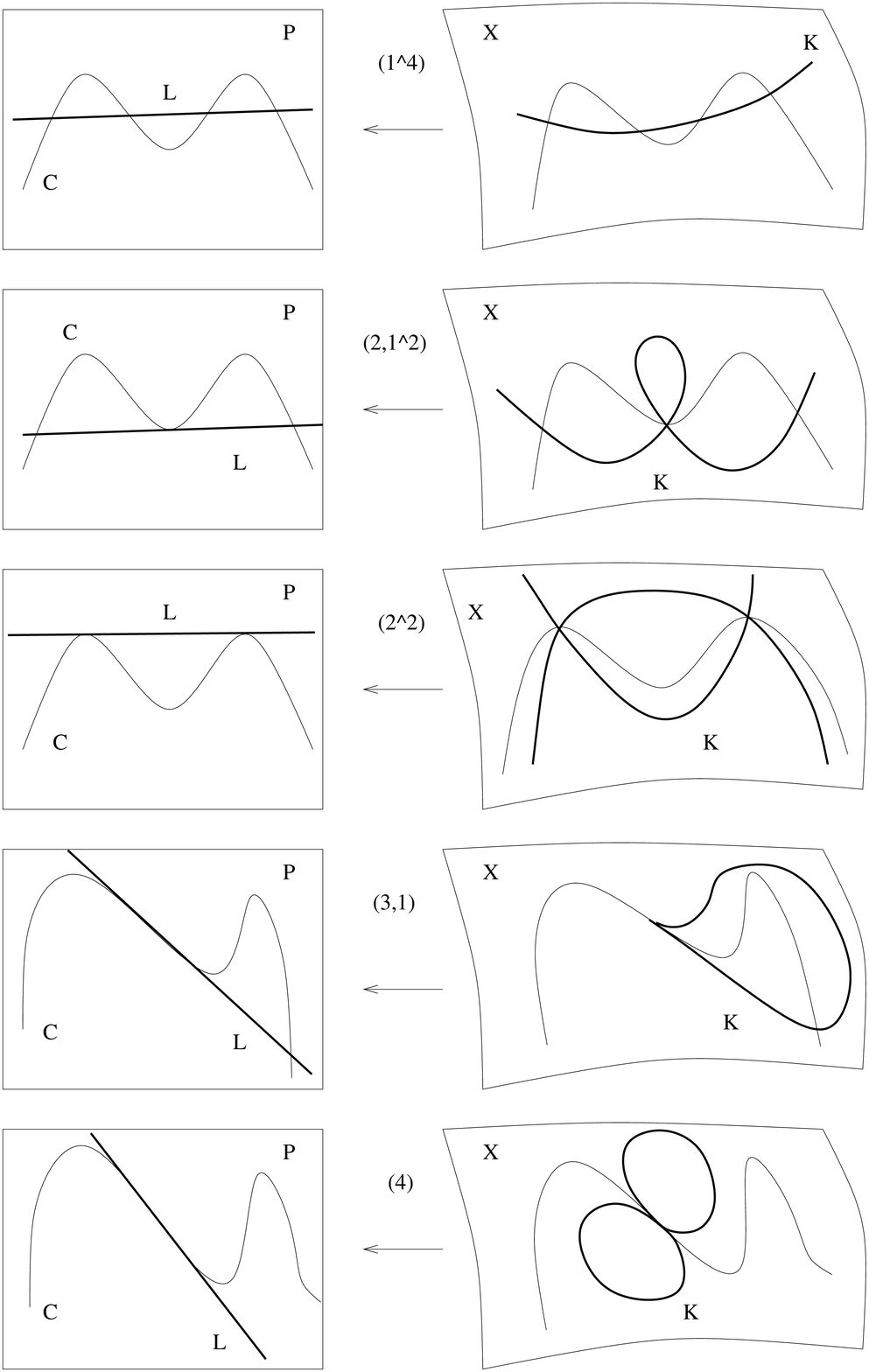}

\caption{{\small The non-hyperelliptic strata and the cover
$X_C\rightarrow P_C$.}}\label{abelian-strata:fig}
\end{figurehere}
\end{center}

\subsection*{Bringing in a canonical divisor}
Let $D$ be a positive canonical divisor on $C$.
Then $D$ is the trace of a line $L_D\subset P_C$ on $C$ and $\Ltil_D:=\pi^*L_D$ is an anticanonical divisor on $X_C$. As we noted, the latter is of arithmetic genus one: if $D$ is general (of type $(1,1,1,1)$), then $\Ltil_D$ is smooth and in the other cases $\Ltil_D$ is a nodal curve 
($D$ is of type $(2,1,1)$), a bigon  (type $I_2$ in the Kodaira classification and $D$ is of type $(2,2)$), a cuspidal curve (type $II$ in the Kodaira classification and $D$ is of type $(3,1)$) or two smooth rational curves meeting in a point with multiplicity 2 (type $III$ in the Kodaira classification and $D$ is of type $(4)$). We regard $\Ltil_D$ as a genus one curve endowed with a polarization of degree 2.

The group $\pic (\Ltil_D)$ of isomorphism classes of line bundles on $\Ltil_D$  has as its identity component $\pic^0 (\Ltil_D)$  an elliptic curve, is isomorphic to $\CC^\times $ or is the additive group $\CC$,
according to whether  $D$ is reduced, has a point of multiplicity $2$, or has a point of multiplicity $\ge 3$.  

The orthogonal complement in $\pic (X_C)$ of the classes of the irreducible components of $\Ltil_D$, denoted here by $\pic (X_C,\Ltil_D)\subset \pic (X_C)$,  is generated by the roots contained in
it, so that these roots make up a root subsystem $\root(X_C, \Ltil_D)\subset \root(X_C)$. We have  only two cases: for $\Ltil_D$ irreducible, we have  of course $\root(X_C,\Ltil_D)=\root(X_C)$ 
and otherwise (when $\Ltil_D$ has two irreducible components  interchanged by $\iota$), $\root(X_C,\Ltil_D)$ 
is of type  $E_6$ (see Table 1).

\subsection*{The basic invariant}
The natural homomorphism $\pic(X_C)\to \pic (\Ltil_D)$ has an evident restriction 
\[
\chi_{C,D}: \pic(X_C,\Ltil_D)\to \pic^0 (\Ltil_D).
\] 
It plays a central role in what follows. Let us first observe that no root $\alpha\in \root(X_C, \Ltil_D)$  lies in the kernel of $\chi_{C,D}$. For such a root $\alpha$ can be represented by a difference $E-E'$ of disjoint exceptional curves $E,E'$ which meet the \emph{same} component of $\Ltil_{D,\reg}$, in $p$ resp.\ $p'$, say.
Then clearly, $(p)-(p')$ represents $\chi_{C,D}(\alpha)$ and since $p\not=p'$, it is a nonzero element of 
$\pic^0 (\Ltil_D)$.\\

\begin{center}
\begin{tablehere}
\begin{tabular}{c|c|c}
$\qquad\kbold\qquad$ & $\quad$Kodaira type$\quad$ & type of  $\root(X_C,\Ltil_D)$ \\
[0.5ex]
\hline
$(1^4)$ & smooth & $E_7$      \\
[1ex]
$(2,1^2)$ & $I_1$ (mult) & $E_7$  \\
[1ex]
$(2^2)$ & $I_2$ (mult)& $E_6$      \\
[1ex]
$(3,1)$ & II (add) & $E_7$  \\
[1ex]
$(4)$ & III (add) & $E_6$
\end{tabular}
\vspace{0.3cm}

\caption{{Nonhyperelliptic strata in genus 3.}}\label{tbl:tablelabel}
\end{tablehere}
\end{center}

We now view $\chi_{C,D}$ as an element of $\pic(X_C,\Ltil_D)^\vee\otimes\pic^0 (\Ltil_D)$. This last group 
is a weight lattice of type $E_6$ or $E_7$ tensored with either an elliptic curve, a copy of $\CC^\times$ or a copy of $\CC$ (which is like a six- or sevenfold power of the latter but with a Weyl group symmetry built in). 
Its isomorphism type is a complete invariant of  the pair $(X_C ,\Ltil_D)$ (and hence of the pair $(C,D)$). To be more concrete, let us identify $\root(C,D)$ with a fixed root system $\root$ of type $E_6$ or $E_7$.
Two such identifications differ by an element of the automorphism group 
$\aut (\root)$
of $\root$, which is $\{\pm 1\}.W(\root)$ in the $E_6$ case and $W(\root)$ in the $E_7$ case.  We also identify  $\pic^0 (\Ltil_D)$ with a fixed group $G$, which is either an elliptic curve, or
the multiplicative group $\CC^\times$ or the additive group $\CC$.  Two such identifications differ by an automorphism of $G$. Notice that $\aut (G)$ equals $\{\pm 1\}$ for $G$ a generic elliptic curve or $G\cong\CC^\times$ and is equal to  $\CC^\times$ when $G\cong\CC$. So if
we denote by $Q(\root)$ the (root) lattice spanned by $\root$,  and by $\Hom (Q(\root), G)^\circ$ homomorphisms with no root in their kernel, then our $\chi_{C,D}$ defines an element of  
\[
S(\root,G):=\aut(\root)\bs \Hom (Q(\root), G)^\circ/\aut (G),
\]
for $\root$ and $G$ as listed. This construction also makes sense for families of elliptic curves. In fact, if $\Ccal_{1,1}/\Mcal_{1,1}$ is the universal elliptic curve (an orbifold), then we can  form  $S(\root, \Ccal_{1,1}/\Mcal_{1,1})$. 
Let us write $\Mcal_{0,(4)}$ for the moduli  stack of $4$-element subsets of $\PP^1$ up to projective equivalence,
in other words, $\Mcal_{0,(4)}:=\Sfrak_4\bs\Mcal_{0,4}$. There is an evident  $\ZZ/2$-gerbe $\Mcal_{1,1}\to  \Mcal_{0,(4)}$ (a `$B(\ZZ/2)$-fibration'), that makes $S(\root, \Ccal_{1,1}/\Mcal_{1,1})$ fiber over 
$\Mcal_{0,(4)}$. 
\begin{theorem}\label{thm:non-hyperelliptic}
The map $(C,D)\mapsto \chi _{C,D}$ induces orbifold  isomorphisms
\begin{align*}
\PP\Hcal(4)&\cong S(E_6,\CC),\\
\PP\Hcal(3,1)&\cong S(E_7,\CC),\\
\PP\Hcal(2^2)-\PP\Hcal(2^2)_\hyp &\cong  S(E_6, \CC^\times),\\
\PP\Hcal(2,1^2)-\PP\Hcal (2,1^2)_\hyp&\cong S(E_7, \CC^\times),\\
\intertext{and an $\Mcal_{0,(4)}$-isomorphism of orbifolds} 
\PP\Hcal(1^4)-\PP\Hcal(1^4)_\hyp&\cong S(E_7, \Ccal_{1,1}/\Mcal_{1,1}).
\end{align*}
Here $\PP\Hcal(\kbold)_\hyp\subset \PP\Hcal(\kbold)$ denotes the locus in $\PP\Hcal(\kbold)$ for which the underlying curve is hyperelliptic.
\end{theorem}

The main result of Deligne in \cite{deligne}  implies that any variety of the form $S(\root,\CC)$ is an orbifold classifying space of its
orbifold fundamental group. If $\root$ has the property that its automorphism group is $\{\pm1\}.W(\root)$ or $W(\root)$
(which is the case when it is of type $A$ or $E$), then the orbifold fundamental group in question equals the quotient of the Artin group of type $W(\root)$ by  its natural (infinite cyclic) central subgroup. Hence we find:

\begin{corollary}\label{cor:}
The stratum $\PP\Hcal(4)$ resp.\ $\PP\Hcal(3,1)$ is an orbifold classifying space for the Artin group of type $E_7$ resp.\ $E_6$ 
modulo its natural (infinite cyclic) central subgroup. 
\end{corollary}

But we do not know how to characterize any of these groups as a kind  a mapping class group.

\section{The hyperelliptic locus}\label{sect:hyperelliptic}
The strata corresponding to $\kbold=(1^4), (2,1^2),(2^2)$
contain both non-hyperelliptic and hyperelliptic curves.
Thus, we must know how adding $\PP\Hcal(\kbold)_\hyp$ to
$\PP\Hcal (\kbold)-\PP\Hcal(\kbold)_\hyp$ can be expressed in terms of the right hand side in Theorem~\ref{thm:non-hyperelliptic}.
Let us first give each of the $\PP\Hcal(\kbold)_\hyp$ a description in the same spirit as the varieties we dealt with.\\

\begin{center}
\begin{figurehere}
\psfrag{L}{$L_D$}
\psfrag{P}{$P_C$}
\psfrag{C}{$C$}
\psfrag{(2^2)}{$(2^2)$}
\psfrag{(2,1^2)}{$(2,1^2)$}
\psfrag{(1^4)}{$(1^4)$}
\includegraphics[width=\textwidth]{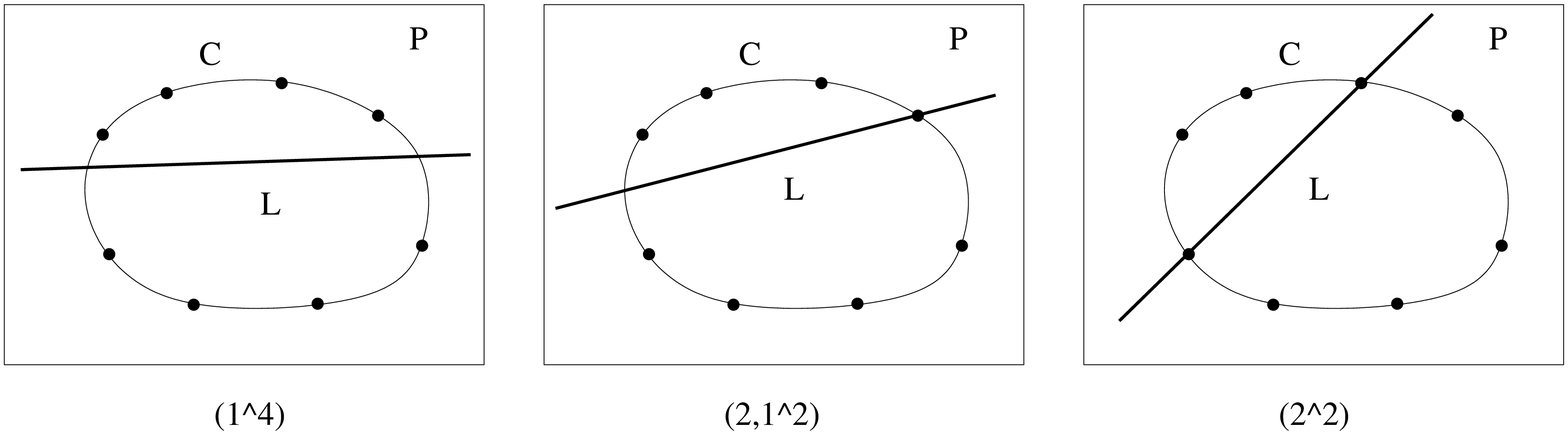}
\caption{{\small Hyperelliptic degenerations in the strata $(1^4),(2,1^2),(2^2)$.}}\label{abelian-strata-hyperelliptic:fig}
\end{figurehere}
\end{center}

Let us first observe that the hyperelliptic involution gives each $\PP\Hcal(\kbold)_\hyp$ the structure of
a $\ZZ/2$-gerbe. The following proposition identifies the base:

\begin{proposition}\label{prop:hypergerbe}
The hyperelliptic involution yields natural $\ZZ/2$-gerbes   
\begin{align*}
\PP\Hcal(1^4)_\hyp &\to S(A_7,\CC^\times),\\
\PP\Hcal(2,1^2)_\hyp &\to W(A_6)\bs \Hom (Q(A_6),\CC^\times)^\circ,\\ 
\PP\Hcal(2,2)_\hyp &\to S(A_5,\CC^\times), 
\end{align*}
where we note that   $W(A_6)\bs\Hom (Q(A_6),\CC^\times)^\circ\to S(A_6,\CC^\times)$ is a double cover.
\end{proposition}
\begin{proof}
The loci in question are the moduli stacks for the pairs $(C,D)$ for which $C$ is a hyperelliptic genus three curve and $\supp (D)$ is the 
union of two distinct orbits under the hyperelliptic involution. The three cases correspond to having 0,1, or 2 Weierstra\ss\ points in $\supp (D)$.
If we divide out by the hyperelliptic involution, we get a copy of $\PP^1$. We make  the identification in such a manner that the 
$\supp (D)=\{ 0,\infty\}$, where in case $(2,1^2)$ we let $\infty$ be the image of the Weierstra\ss\ point. This identifies  the stratum 
$\PP\Hcal^\hyp(1^4)$,  $\PP\Hcal^\hyp(2,1^2)$,  $\PP\Hcal^\hyp(2,2)$ modulo the hyperelliptic involution with the  configuration space of subsets of $\CC^\times $ of  8 resp.\ 7 resp.\ 6 elements, modulo the obvious $\CC^\times$-action and, in the first and the last case, also modulo the involution. These are easily seen to be as asserted.
\end{proof}

In order to understand how these loci lie in their strata, we need to have a unified picture that includes both non-hyperelliptic and
hyperelliptic curves.
While the anti-canonical model is adequate to discuss the case of a non-hyperelliptic curve, the bi-anti-canonical model is more suited to analyze what happens near the hyperelliptic locus.

The degeneration of the del Pezzo surfaces using
the anti-canonical model only was analyzed in \cite{looij2008}.

\section{The bicanonical model}\label{sect:bianti}
This section is inspired by similar ideas occurring in
\cite{shah1980}.

Given a curve $C$ of genus $3$, then the canonical map  
\[
k: C\to P_C:=\check\PP (H^0(C,\can_C))
\]
is an embedding in a projective plane unless $C$ is hyperelliptic, in which case it factors through the  hyperelliptic involution (which we will always denote by $\tau$) and has a conic as image. On the other hand, the bicanonical map 
\[
k': C\to P'_C:=\check\PP (H^0(C,\can^{\otimes 2}_C))
\]
is always an embedding in a 5-dimensional projective space and this allows us to identify $C$ with its image in $P'_C$. 
Let 
\[
j :P_C\hookrightarrow \check\PP(\Sym^2 H^0(C,\can_C))
\]
be the Veronese embedding. The multiplication map 
\[
m:\Sym^2H^0(C,\can_C)\rightarrow H^0(C,\can^{\otimes 2}_C)
\]
induces a rational map $[m^*]: P'_C\dashrightarrow \check\PP(\Sym^2 H^0(C,\can_C))$ and $C$ will always lie on $V_C:=[m^*]^{-1}j(P_C)$. It is therefore of interest to determine what $V_C$ is like.

\subsection*{Non-hyperelliptic case}
For $C$ non-hyperelliptic the map $m$ is an isomorphism. Hence so is $[m^*]$ and this shows that $V_C$ is a Veronese surface  and naturally isomorphic to $P_C$. We also find that this isomorphism takes $k(C)$ to $k'(C)$. Thus, $k'$ embeds $C$ in $P'_C$ as the intersection of $V_C$ with a quadric: $k'(C)$ appears as a divisor of $\Ocal_{V_C}(2)$.

\subsection*{Hyperelliptic case}
The situation is somewhat more complicated for $C$ hyperelliptic.
The hyperelliptic involution $\tau$ acts as $-1$ in $H^0(C,\can_C)$ and hence trivially in $\Sym^2 H^0(C,\can_C)$. But it acts in $H^0(C,\can^{\otimes 2}_C)$ as reflection. The image of $m$ is the full fixed point hyperplane of this reflection and so $m$ has $1$-dimensional kernel and cokernel. The kernel of $m$ is of course spanned by an element of $\Sym^2 H^0(C,\can_C)$, which, when viewed as a quadratic form on
$H^0(C,\can_C)^*$, defines the image of $k$ (a conic in $P_C$). And so the hyperplane $H_\infty\subset\check\PP(\Sym^2 H^0(C,\can_C))$ defined by this kernel has the property that $jk(C)=H_\infty\cap j(P_C)$.
It also follows that the fixed point set of $\tau$ in $P'_C$ consists of a hyperplane $H$ and a singleton $\{v\}$, that $[m^*]$ establishes an isomorphism $H\cong H_\infty$ (we will therefore identify the two) and that $[m^*]:P'_C\rightarrow\check\PP(\Sym^2 H^0(C,\can_{C}))$ is the linear projection with center $v$ onto $H$ followed by the embedding $H\hookrightarrow \check\PP(\Sym^2 H^0(C,\can_C))$. This implies that $V_C$ is the cone with vertex $v$ and base the image of $jk(C)\subset H$. The image $k'(C)$ lies in the smooth part of the cone $V_C$ and (again)  appears as a divisor of $\Ocal_{V_C}(2)$.

Notice that the involution $\tau$
acts on each ray of $V_C$ as the unique nontrivial involution that fixes its intersection with $H$ and the vertex $v$. Moreover, the linear projection maps $k'(C)$ onto as a double cover
branched along an 8-element subset $B=k'(C)\cap jk(C)$, whose two sheets are exchanged by $\tau$.

Incidentally, we remark that there are two distinct types
of rays $\R\subset V_C$:
\begin{itemize}
\item[(a)]
an {\it ordinary ray} $\R$: one that meets $k'(C)$ transversely in two points, so that $\tau$ acts nontrivially on $\R\cap k'(C)$;
\item[(b)]
a {\it Weierstra\ss\ ray} $\R=\R_b$: one  that is
tangent to $k'(C)$ at a point $b\in B=k'(C)\cap H_\infty$.
\end{itemize}

\begin{center}
\begin{figurehere}
\psfrag{v}{$v$}
\psfrag{k(C)}{$k'(C)$}
\psfrag{k'(C)}{$k(C)$}
\psfrag{R}{$\R$}
\psfrag{Rb}{$\R_b$}
\psfrag{b}{$b$}
\psfrag{H}{$H$}
\psfrag{V}{$V_C$}
\includegraphics[width=0.9\textwidth]{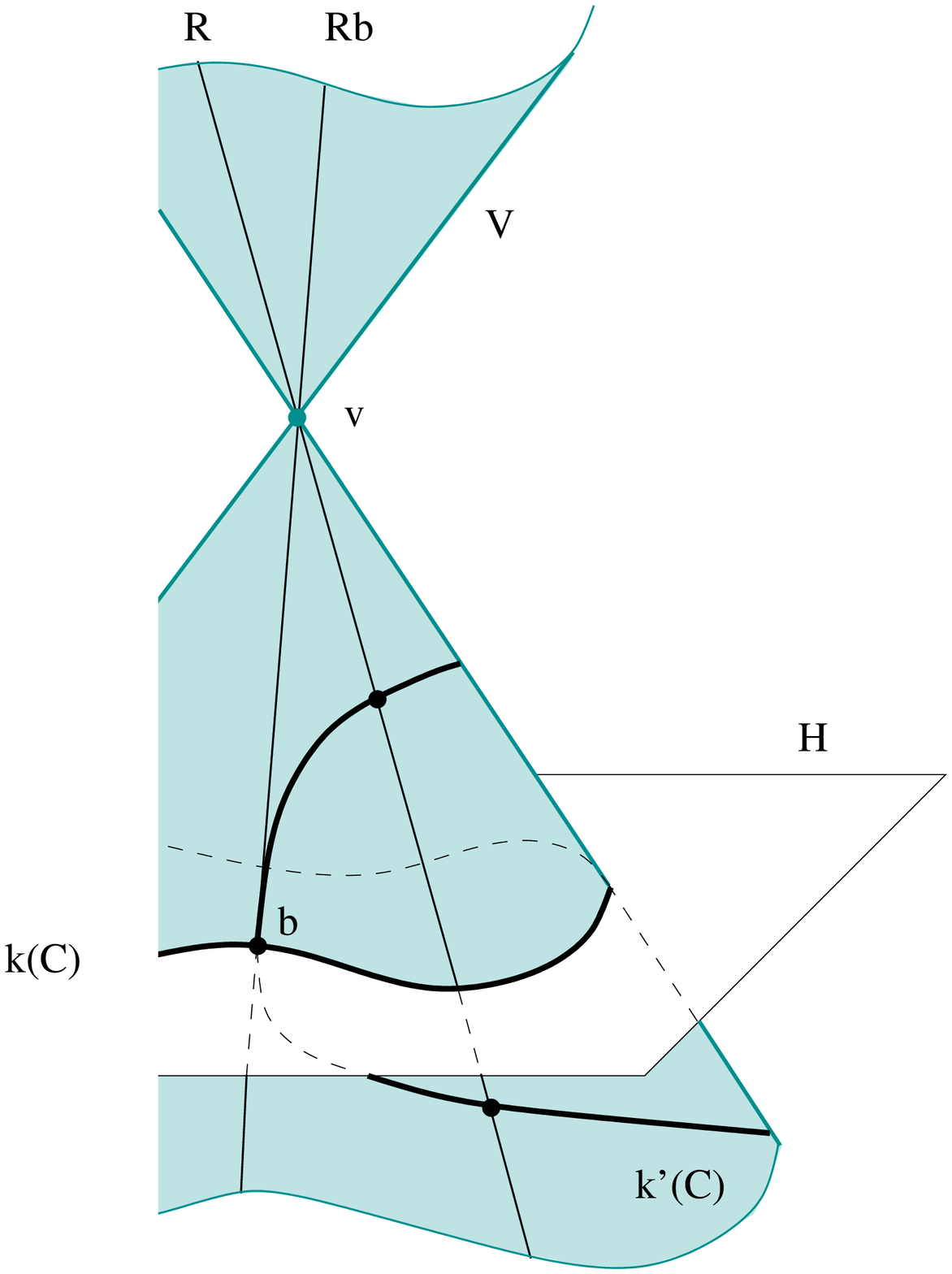}

\caption{{\small The quartic cone $V_C$ for a hyperelliptic $C$.}}\label{cone:fig}
\end{figurehere}
\end{center}

We also observe that the minimal resolution of
the vertex of $V_C$ has as its exceptional set a rational curve of self-intersection $-4$.
The complement of the vertex in $V_C$ is isomorphic to the total space
of the line bundle $\Ocal_{\PP^1}(4)$ over $\PP^1\cong k(C)$.
In other words, $V_C$ is obtained from a Segre-Hirzebruch surface $\Sigma_4$ by blowing down the $(-4)$-section. Thus, $V_C$ is a simply connected rational homology manifold with the homology
of a complex projective plane
%
and one can directly check that the restriction map
$H^2(P'_C;\ZZ)\rar H^2(V_C;\ZZ)$ is an isomorphism.

Indeed, the $(-4)$-singularity of the cone $V_C$ is homogeneous 
and $V_C\setminus H$ can be obtained as the quotient of an affine plane $\AA^2$
by the diagonal action of $\mu_4$. The local Picard group of a
$(-4)$-singularity is cyclic of order $4$ with generator a ray $\R$
and $4\R$ is a hyperplane section of $V_C$ and so locally principal.
It follows that $\pic(V_C)$ is generated by the class of $4\R$, and that since 
$(4\R)\cdot \R= H\cdot \R=1$, it follows that $\R^2=1/4$. Clearly, $C\cdot \R=2$.

Conversely, the pair $(V_C,C)$ is easily reconstructed up to isomorphism from an
8-element subset $B$ of a projective line $\PP^1$, for the double cover of $\PP^1$ ramified along $B$ sits naturally in the total space
of $\Ocal_{\PP^1}(4)$ and that total space can be identified with the complement of the vertex of a cone as above.

{\it Notation.}
In what follows, we will always identify a curve $C$ with its image through the bicanonical embedding. For a hyperelliptic $C$, we will
write $\Ccan$ for the rational curve $jk(C)$ sitting inside $V_C$.

\subsection*{Near the hyperelliptic locus}
Let $C_0$ be  hyperelliptic of genus 3. It has a semi-universal deformation whose base is a smooth germ $(S,o)$ of dimension 6 with $T_oS$ naturally identified with the dual of $H^0(C_0,\omega_{C_0}^{\otimes 2})$. The universal property 
implies that the hyperelliptic involution $\tau$ extends to this universal deformation. The identification 
$T_oS\cong H^0(C_0,\omega_{C_0}^{\otimes 2})^*$ is  $\tau$-equivariant. As we have seen  $\tau$ acts as a reflection on $H^0(C_0,\omega_{C_0}^{\otimes 2})$ and so $\tau$ also acts as such on  $T_oS$. The action of $\tau$ on $(S,o)$ can be linearized in the sense that we can choose coordinates for $(S,o)$ in terms of which  $\tau$ is a reflection. The $(+1)$-eigenspace then parametrizes  hyperelliptic deformations. But our interest  concerns rather the $(-1)$-eigenspace, or more invariantly put, a transversal  slice to the hyperelliptic locus invariant under $\tau$ (it is not unique). In other words, we consider a one-parameter deformation $f:\Ccal\rar\Delta$ of $C_0$ such that  $C_t$ is non-hyperelliptic for $0\neq t\in\Delta$, that the Kodaira-Spencer map is nonzero at $t=0$ 
and that the hyperelliptic  involution $\tau$ extends to our family and takes $C_t$ to $C_{-t}$.

The above construction carries over in families and we obtain a factorization of the relative bicanonical embedding over $\Delta$
\[
 \Ccal/\Delta\hookrightarrow\Vcal/\Delta\subset \Pcal'/\Delta
\]
through a surface $\Vcal$ over $\Delta$ with central fiber $V_0$ a cone over a rational rational normal curve and general fiber a Veronese surface and we obtain $\Ccal$ as divisor of a section $s$ of  $\Ocal_\Vcal(2)$. 
We incidentally notice that $V_0\setminus H$ deforms to the complement of a hyperplane section of $V_t$, and so the Milnor fiber
of this degeneration has the homotopy type of the complement of a conic in $\PP^2$, that is, of a real projective plane. 

We note that $\pic (\Vcal)\cong \pic (V_0)$ is generated by $\Ocal_{\Vcal}(1)$. But on $V_t$, $t\not=0$,  $\Ocal_{V_t}(1)$ represents twice the generator and so the image of the specialization map
$\pic(V_0)\cong \pic (\Vcal) \rar \pic(V_t)$ has index two. This corresponds (dually) to the fact that a line on the generic fiber
of $\Vcal/\Delta$ (that is, a line in $\Pcal|\Delta^*$ under the image of Veronese map) may only  extend as a Weil divisor with specialization  to the sum of two rays (only four times a ray is a divisor of $\Ocal_{V_t}(1)$). For example, each bitangent  of 
$\Ccal |\Delta^*$ specializes to the sum of two Weierstra\ss\ rays $\R_b+\R_{b'}$ and this establishes a bijection between the 28 bitangents of $\Ccal |\Delta^*$ and the collection of all 2-element subsets of $B$.

We deduce  that the monodromy action of $\pi_1(\Delta^*)$ on the cohomology of the general fiber of $\Vcal/\Delta$ is trivial.
Indeed, the bundle $\Pcal'_\Delta$ is trivial and $H^*(P'_t;\QQ)\rar H^*(V_t;\QQ)$ is an isomorphism
for all $t\in\Delta$. For essentially the same reason, the  ``half-monodromy'' action of $\tau$: 
$H^*(V_t;\ZZ)\rar H^*(V_{-t};\ZZ)$ is (after we make an identification $H^*(V_{-t};\ZZ)\cong H^*(V_{t};\ZZ)$ via a path in $\Delta^*$) the also trivial.

\subsection*{The double cover construction}
The action of $\tau$ on the family $\Ccal/\Delta$
and so on the total space of the vector bundle $f_*(\can_f^{\otimes 2})$ induces an action on the total space of $\Ocal_{\Vcal}(1)$, which in particular fixes the fibers over $\Ccan_0$. It acts
as minus the identity on the fiber over $v$ and as plus the identity on its restriction to $H\cap V_0$. Hence, upon replacing $s$ by $s+\tau^*s$, we may assume that $s$ is $\tau$-invariant.
As $s$ is a divisor for $\Ccal$, we can construct
the double cover $\pi:\Ycal\rightarrow\Vcal$ branched over $\Ccal$
inside the total space of $\Ocal_{\Vcal}(1)$ as the set of points whose square is a value of $s$.
By construction it comes with an action of $\tau$. If we identify $\Vcal$ with the zero section of this total space, then we see
that $\Ccal$ is also the ramification locus of $\pi$. We denote by $\iota$ the natural involution of $\Ycal/\Vcal$.
We remark that the actions of $\iota$ and $\tau$ commute.

For $t\neq 0$, $Y_t$ is clearly
isomorphic to the del Pezzo surface $X_{C_t}$, whereas
$Y_0$ has two isolated singular points $v^+,v^-$ mapping to the vertex $v$ of $V_0$.
Notice that besides the inclusion $C_0\subset \Ccal$ there is another copy of $C_0$ embedded inside $Y_0$,
namely preimage $\pi^{-1}(\Ccan_0)$ and that the two meet each other in
the 8 Weierstra\ss\  points. The involution $\iota$ fixes $C_0$ pointwise and acts as the hyperelliptic involution on $\pi^{-1}(\Ccan_0)$, whereas $\tau$ fixes $\pi^{-1}(\Ccan_0)$ pointwise and acts nontrivially on $C_0$.

Denote by $\Rtil$ be the preimage inside $Y_0$
of a ray $\R\subset V_0$.

\begin{itemize}
\item[(a)]
If $\R$ is an ordinary ray, then
$\Rtil$ is a smooth rational curve that doubly covers $\R$.
There are six special points on $\Rtil$: the ramification points
$r_1,r_2$, the singularities $v^+,v^-$ of $Y_0$ and
the preimages $h^+,h^-$ of $\R\cap H_\infty$.
The involutions $\iota$ and $\tau$ on $\Rtil$ are characterized by the property that their fixed point
pairs are $\{ r_1,r_2\}$ and $\{h^+,h^-\}$ respectively (so $\iota$, $\tau$ and $\iota\tau$ permute
the 6 points respectively as
$(v^+\, v^-)(h^+\, h^-)$, $(v^+\, v^-)(r_1\, r_2)$ and $(h^+\,h^-)(r_1\,r_2)$.
\item[(b)]
If $\R_b$ is a Weierstra\ss\ ray through $b\in B$,
then $\Rtil_b=R_b^+\cup R_b^-$, where
$R_b^\pm$ is a smooth rational curve passing through $v^\pm$ and isomorphically mapping to $\R_b$
and $R_b^+\cap R_b^-=\{b\}$ ($b$ an ordinary double point of $\Rtil_b$).
Both $\iota$ and $\tau$ act on $\Rtil_b$ by exchanging $v^+$ and $v^-$ and fixing $b$, and so exchanging $R_b^+$ and $R_b^-$; but their action is different: in fact, $\iota\tau$
preserves each component and acts on $R_b^\pm$ as the only
nontrivial involution that fixes $v^\pm$ and $b$.
\end{itemize}

The variety $Y_0$ is a rational homology manifold too, and indeed
the specialization map $H^*(Y_0;\ZZ)\rar H^*(Y_t;\ZZ)$ is
an isomorphism, so that Poincar\'e duality holds
and all divisors are $\QQ$-Cartier.

The intersection pairing of the liftings of Weierstra\ss\ rays
satisfy $R_b^+\cdot R_b^-=1$ and
$(R_b^\pm)^2=-3/4$ for every $b\in B$ and that for  $b,b'\in B$ distinct,
$R_b^+\cdot R_{b'}^-=0$ and $R_b^\pm\cdot R_{b'}^\pm=1/4$.
So $R_b^\pm -R_{b'}^\pm$ is a Cartier divisor of self-intersection $-2$, whereas $R_b^\pm +R_{b'}^\pm$ is 
only a Weil divisor and has self-intersection $-1$.

\begin{center}
\begin{figurehere}
\psfrag{v+}{$v_+$}
\psfrag{v-}{$v_-$}
\psfrag{C}{$C_0$}
\psfrag{C'}{$\pi^{-1}(\Ccan_0)$}
\psfrag{R}{$\tilde{\R}$}
\psfrag{R+}{$\R_b^+$}
\psfrag{R-}{$\R_b^-$}
\psfrag{r1}{$r_1$}
\psfrag{r2}{$r_2$}
\psfrag{h+}{$h_+$}
\psfrag{h-}{$h_-$}
\psfrag{VC}{$Y_0$}
\psfrag{bt}{$b$}
\includegraphics[width=0.7\textwidth]{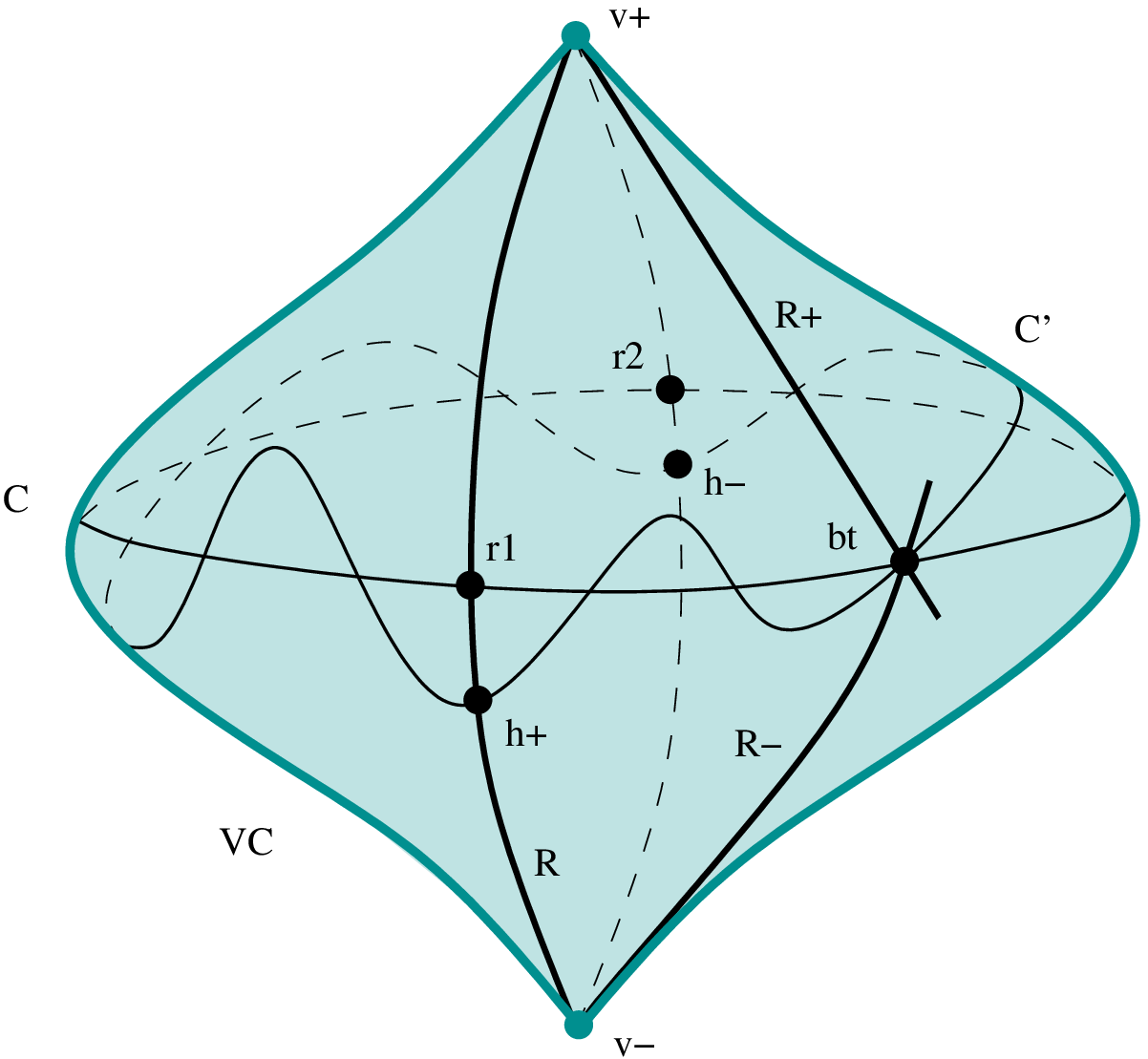}

\caption{{\small The double cover $Y_0$ of the quartic cone $V_0$.}}\label{double-cone:fig}
\end{figurehere}
\end{center}

\subsection*{Limits of exceptional classes}
As anticipated before, the 28 double tangents
appear in the limit in the central fiber of $\Vcal_\Delta$ as the $\binom{8}{2}$ couple of Weierstra\ss\ rays.
The corresponding pairs of exceptional curves on the general fiber have a decent limit on the central fiber of $\Ycal\to \Delta$ 
and this limit can be described as follows. For every $2$-element  subset $\two=\{b,b'\}\subset B$, denote by $E^+_\two$ resp. by $E^-_\two$
the Weil divisor $R^+_b+R^+_{b'}$ resp. $R^-_b+R^-_{b'}$
inside $Y_0$.
This describes the 56  limits of the exceptional curve on the general fiber: we get a flat family $\Ecal^\pm_\two/\Delta$ whose general fiber is an exceptional curve on a del Pezzo surface of degree 2. Notice that the involutions $\iota$ and $\tau$ interchange $\Ecal^+_\two$ and  $\Ecal^-_\two$. The intersection numbers of these exceptional classes, insofar they are not self-intersections, take their values in $\{ 0,1,2\}$ and we have (compare \cite{looij2008}, Lemma (6.4) for a different model): 
\begin{itemize}
\item[(i)] if $\two\cap \two'$ is a singleton, then
$E_{\two,t}^\pm\cdot E_{\two',t}^\pm =0$ and $E_{\two,t}^\pm\cdot E_{\two',t}^\mp =1$,
\item[(ii)] if $\two\cap \two'=\emptyset$, then
$E_{\two,t}^\pm\cdot E_{\two',t}^\pm =1$ and $E_{\two,t}^\pm\cap E_{\two',t}^\mp =0$ and
\item[(iii)] $E_{\two,t}^+\cdot E_{\two,t}^-=2$.
\end{itemize}
Since the Picard group of the general fiber of $\Ycal/\Delta$ is spanned by the exceptional classes, we see how that group specializes in the central fiber.  We also see that  on the central fiber,  roots are still represented as differences  of disjoint exceptional curves.  

Notice that this construction leaves a trace on the relative Picard group of the generic fiber $\Ycal^*/\Delta^*$ (the superscript ${}^*$ refers to restriction over  $\Delta^*$), for it divides the exceptional curves of the generic fiber in two subsets that are interchanged by the involution $\iota$. These are in fact separated by the character $\eps: \pic (\Ycal^*/\Delta^*)\to \{ \pm 1\}$ which takes the value $\pm 1$ on $\Ecal_{\two}^\pm$.  In fact, we may identify $\pic(\Ycal/\Delta)$ with the kernel of $\eps$.

We can be more explicit if $\Ycal^*/\Delta^*$ is  given as a projective plane $\PP^2_{\Delta^*}$ blown up in $7$ numbered $\Delta^*$-valued points. This yields another basis of $\pic(\Ycal^*/\Delta^*)$, namely 
$\ell, e_1,\dots ,e_7$ , where $\ell$ is the preimage of the class of a line in $\PP^2_{\Delta^*}$ and $e_i$ denotes the class of the exceptional divisor over the $i$-th
point.  Then $\Kcl=3\ell-\sum_{i=1}^7 e_i$ is the anticanonical class and we can number the elements of $B$: $B=\{b_0,\dots ,b_7\}$ in such a manner that  
\begin{align*}
e_i&= \Ecal_{b_0,b_i}^-, \quad 1\le i\le 7,\\
\Kcl-e_i&=\Ecal_{b_0,b_i}^+,\quad 1\le i\le 7,\\
\ell-e_i-e_j&=\Ecal_{b_i,b_j}^+,\quad 1\le i<j\le 7,\\
\Kcl-(\ell-e_i-e_j)&=\Ecal_{b_i,b_j}^-,\quad 1\le i<j\le 7.\\
\end{align*}
The elements $\ell-e_1-e_2-e_3$ and $\{e_{i-1}-e_{i}\}_{i=2}^7$ make up a root basis of $\root(\Ycal^*/\Delta^*)$. 
The roots in $\ker (\eps)$ are $e_i-e_j$, $1\le i<j\le 7$ and $\pm(\Kcl-\ell +e_i)=\pm(2\ell -e_1\cdots -\widehat{e_i}\cdots -e_7)$, $i=1,\dots ,7$.
This is in fact a root subsystem of type $A_7$ for which $(-2\ell +e_2\cdots +e_7,e_1-e_2,\dots , e_6-e_7)$ is a root basis.

We rephrase this for purposes of record in the following.

\begin{proposition}
Any root $\alpha\in \root(\Ycal^*/\Delta^*) $ can be written either as a difference $\Ecal_{\two}^\pm-\Ecal_{\two'}^\pm$ (the roots in $\Ker (\eps)$)  or as $\Ecal_{\two}^\pm-\Ecal_{\two'}^\mp$ (the remaining ones), depending on whether $\two\cap \two'$ is a singleton or empty. The roots  of the first type are precisely the ones that lie in the image of $\pic(\Ycal/\Delta)$ and make up a root subsystem $\root_0(\Ycal/\Delta)$ of $\root(\Ycal^*/\Delta^*)$ of type $A_7$. Moreover, the permutation action of $\perm (B)$ on the  collection of $\Ecal^\pm_\two$'s defines an isomorphism of $\perm (B)$ onto the Weyl group of $\root_0(\Ycal/\Delta)$ and identifies $\perm (B)\times\la \iota\ra$ with the $W(\Ycal^*/\Delta^*)$-stabilizer of $\Ker(\eps)$. Both $\tau$ and $\iota$ act on $\root(\Ycal^*/\Delta^*)$ as minus the identity.
\end{proposition}

Now let be given a family of conics $L_t\subset V_t$ degenerating into $L_0=\R'+\R''\subset
V_0$.
There are three cases:
\begin{itemize}
\item
$\R'$ and $\R''$ are ordinary rays, accounting for $\PP\Hcal(1^4)_\hyp$, 
\item
$\R'$ is ordinary and $\R''$ is Weierstra\ss,  accounting for $\PP\Hcal (2,1^2)_\hyp$, 
\item
both $\R'$ and $\R''$ are Weierstra\ss\ rays, accounting for $\PP\Hcal(2,2)_\hyp$.  
\end{itemize}
We think of $\Lcal$ as defining a relative canonical divisor $\Dcal/\Delta$ on the degenerating family $\Ccal/\Delta$ and we need to understand how the basic invariant $\chi_{\Ccal,\Dcal}$ specializes over $0\in\Delta$. This is the subject of the following two sections.

\section{The open stratum}  
This concerns the stratum $\PP\Hcal (1^4)$. We focus on the limiting behavior near the hyperelliptic locus and so we assume that we are in the situation of Section \ref{sect:bianti} and that $L_0$ is the sum of two ordinary rays $\R'$ and $\R''$.
Then the preimage of $L_0$ in $Y_0$ (denote it $\Ltil_0$) consists two smooth rational components
$\Rtil'$ and $\Rtil''$ that meet at $v^+$ and $v^-$.
This is the central fiber of a genus one fibration $\Lcaltil /\Delta$ with smooth general fiber: it is a degeneration of type $I_2$;
in particular, the central fiber is of multiplicative type.
The $j$-function of such a degeneration has a pole of order two at $0\in\Delta$. In fact, the involution $\tau$ nontrivially acts on the family $\Lcaltil$ interchanging the fibers with the same $j$-invariant.

%

\begin{lemma}\label{lemma:specialize}
Let $\alpha\in \root(\Ycal/\Delta)$ be a root.
Then
$\chi_{\Ccal,\Dcal}(\alpha)\in\pic(\Lcaltil/\Delta)$ specializes as
an element
in the identity component of $\pic(\Ltil_0)$
if and only if $\alpha\in \root_0(\Ycal/\Delta)$.
\end{lemma}

\begin{proof}
Let $\Ycalhat$ be obtained by blowing up $\Ycal$ at $v^+$ and $v^-$.
Then $\Yhat_0$ is the union of a double cover $\tilde{\Sigma}_4$ of a Segre-Hirzebruch surface and two copies $\Sigma_0^\pm$ of $\PP^1\times\PP^1$ glued along the two $(-4)$-sections
$\sect^\pm$ of
$\tilde{\Sigma}_4$.
Denote by $\Rhat',\Rhat''$ the strict transforms of $\Rtil',\Rtil''$ and by $\Lcalhat^*$ the preimage of $\Lcaltil^*$
inside $\Ycalhat^*$. Then its closure $\Lcalhat$
is a degeneration of type $I_4$ as $\Lhat_0$ is the union of
$\Rhat'$, $\Rhat''$, $\sect^+$ and $\sect^-$.

If $\Rhat^\pm_b$ is the strict transform of $R^\pm_b$ for
some $b\in B$, then $\Rhat^\pm_b$ meets $\Lhat_0$
in the smooth locus of $\sect^\pm$.
Hence, if $\Ecalhat^\pm_\two\subset \Ycalhat$ denotes the strict transform of $\Ecal^\pm_\two$, then
the locus $\Lcalhat \cap \Ecalhat^\pm_\two$ defines a relative divisor $\delta_\two^\pm$ on $\Lcalhat /\Delta$ of degree $2$.
%

Let $\alpha\in \root(\Ycal/\Delta)$ be represented by $\Ecal_{\two}^\pm-\Ecal_{\two'}^\pm$ resp.
$\Ecal_{\two}^\pm-\Ecal_{\two'}^\mp$.
Then 
$\chi_{\Ccal,\Dcal}(\alpha)\in \pic(\Lcalhat^*/\Delta^*)$ specializes to  $\delta_\two^\pm (0)-\delta_{\two'}^\pm (0)$ resp.\ $\delta_\two^\pm (0)-\delta_{\two'}^\mp (0)$ on $\Lhat_0$.
This specialization lies in the identity component of
$\pic(\Lhat_0)$ precisely if $\eps(\alpha)=1$.
The conclusion follows because the pull-back induces an isomorphism between the
identity components of $\pic(\Lhat_0)$ and $\pic(\Ltil_0)$.
\end{proof}

This  describes in a rather concrete manner how the  restriction homomorphism $\pic (\Ycal^*/\Delta^*)\to \pic (\Lcaltil^*/\Delta^*)$ specializes over the central fiber, for we also find that  the limit $\chi_{\Ccal,\Dcal}(\alpha)$ exists precisely if $\alpha\in \root_0(\Ycal/\Delta)$. Indeed, if $\eps (\alpha)=-1$, then after identifying $\pic^0 (\Ltil_0)$ with $\CC^\times$, the value of  $\chi_{\Ccal,\Dcal}(\alpha)$  tends to $0$ or $\infty$ (depending on the identification). 
The involutions $\iota$ and $\tau$
act on $\pic^0 (\Ltil_0)$ as the inversion and $\iota\circ\chi_{\Ccal,\Dcal}(\alpha)=\tau\circ\chi_{\Ccal,\Dcal}(\alpha)=\chi_{\Ccal,\Dcal}(-\alpha)$.

This gives rise to the following extension of $S(E_7, \Ccal_{1,1}/\Mcal_{1,1})$. We begin with what may be considered as a reconstruction of $\Lcaltil/\Delta$.
We start out with the algebraic torus $\T=(\CC^\times)^2$ and the automorphism $u$ of $\T$ defined by  $u(z_1,z_2)=(z_1,z_1z_2)$.  This automorphism preserves the open subset $\Tcal :=\Delta^*\times\CC^\times$ and generates a group $u^\ZZ$  thats acts properly and freely on $\Tcal$. So the orbit space  $\Fcal^*$ is a complex manifold of dimension two. It maps homomorphically  to $\Delta^*$  and this realizes $\Fcal^*$ as the Tate curve over $\Delta^*$. We  construct an extension $\overline{\Fcal}$ of $\Fcal^*$ over $\Delta$ by means of a familiar construction from toric geometry. 

The coordinates give the lattice $N_T$ of one parameter subgroups of $T$ a natural basis $(e_1, e_2)$. The rays in $N_T\otimes\RR$ spanned by the vectors $e_1+ ne_2$
with $n\in\ZZ$, and the sectors spanned by two successive rays define a partial polyhedral decomposition $\Sigma$ of $N_T\otimes\RR$. This decomposition is clearly invariant under $u$. The associated torus embedding $T\subset T_\Sigma$ is a complex  manifold of dimension two.  Let  $\Tcal_\Sigma\supset \Tcal$ be  the interior of the closure of $\Tcal$ in $T_\Sigma$. Then $u^\ZZ$ acts properly and freely on $\Tcal_\Sigma$. We let $\overline{\Fcal}$ be the orbit space of $\Tcal_\Sigma$ with respect to the subgroup
$u^{2\ZZ}$.
This is also a complex manifold and it
is the total space of a degeneration $\overline{\Fcal}/\Delta$
of curves of genus $1$ of type $I_2$.
We denote by $F_0$ (resp.\ by $\Fcal$) the complement of the two punctual strata in the central fiber inside $\overline{F}_0$ (resp.\ inside $\overline{\Fcal}$).
The section $\sect_0(z_1)=[z_1,1]$ of $\Fcal/\Delta$
makes it into a relative abelian variety, which we denote by
$\Jcal/\Delta$.
We remark that $J_0\cong \pic^0(F_0)\times\{\pm 1\}
\cong \CC^\times\times\{\pm 1\}$.

The automorphism $u$ induces in $\overline{\Fcal}/\Delta$ an order two translation. We have also natural commuting
involutions $\iota$ and $\tau$ on $\overline{\Fcal}/\Delta$, which are
defined as $\iota[z_1,z_2]=[z_1,z_2^{-1}]$ and
$\tau[z_1,z_2]=[-z_1,-z_2^{-1}]$, so that the ``half-monodromy'' acts on $F_t$ as $[z_2]\mapsto [-z_2^{-1}]$.

The induced actions of $\iota$ and $\tau$ on
$\Jcal/\Delta$, given by
$\iota[z_1,z_2]=[z_1,z_2^{-1}]$ and $\tau[z_1,z_2]=[-z_1,z_2^{-1}]$, generate the automorphism group
of $\Jcal/\Delta$.
We incidentally notice that
the ``half-monodromy'' acts on $J_t$ as the inverse (for the
group operation on $J_t$).

If $\Jcal(\Delta)$ denotes  the group of  sections of $\Jcal/\Delta$, then we have a natural surjective homomorphism
\[
c: \Jcal(\Delta)\to J_0\to \{\pm1\}.
\]
Let $\root$ be a fixed root system of type $E_7$. The group of homomorphisms $\chi: Q(\root)\to \Jcal(\Delta)$ is represented by a
$\Delta$-scheme that we shall denote by
$\Hom (Q(\root),\Jcal/\Delta)$.
Concretely,   a basis  $\alpha_1,\dots ,\alpha_7$ of $Q(\root)$ identifies this $\Delta$-scheme with a 7-fold fiber product $\Jcal\times_{\Delta}\Jcal\times\cdots\times_{\Delta}\Jcal$. Its central fiber has $2^7$ connected components and these are canonically labeled by the group
$\Hom(Q(\root),  \{\pm1\})$.
It follows from our discussion of the hyperelliptic limit that we must consider only
some of these components, namely those that correspond to $\chi$ for which
$c\chi :Q(\root)\to  \{\pm1\}$
has as kernel a root sublattice of type $A_7$. 
At this point we recall that the root subsystems of type $A_7$ of $\root$ are transitively permuted by the 
Weyl group $W(\root)$ and that the sublattice spanned by such a subsystem has index 2 in $Q(\root)$. 

Let us denote by $\Hom_{(A_7)}(Q(\root),\Jcal/\Delta)$ the
locus in $\Hom (Q(\root),\Jcal/\Delta)$ defined by the $\chi$ with the above property.
This subset is open and $W(\root)$-invariant. By the preceding remark,  $W(\root)$ is transitive on the connected components of the central fiber, the stabilizer of a connected component being a Weyl group of type $A_7$ times the center $\{\pm1\}$ of $W(\root)$. Removing the fixed point loci of reflections yields an open subset  $\Hom_{(A_7)}(Q(\root),\Jcal/\Delta)^\circ$ and we then put
\[
S_{(A_7)}(E_7,\Jcal/\Delta):= \aut (\root)\bs\Hom_{(A_7)}(Q(\root),\Jcal/\Delta)^\circ/\aut (\Jcal/\Delta).
\]
If we fix a root subsystem $\root_0\subset \root$ of type $A_7$, then we see that the central fiber of $S_{(A_7)}(\root,\Jcal/\Delta)$ is identified with the component of $S(\root, J_0)\cong S(\root, \CC^\times \times \{\pm 1\})$ that maps $\root_0$ to $\CC^\times\times\{1\}$. Restriction to $\root_0$ identifies this in turn with $S(\root_0,\CC^\times)$. We also observe that $S_{(A_7)}(E_7,\Jcal/\Delta)$ maps to
the quotient of $\Delta^*$ by the involution $z_1\mapsto -z_1$. So if we ignore the orbifold structure, then we  have attached a copy of $S (A_7, \CC^\times)$ to $S(E_7, \Ccal_{1,1}/\Mcal_{1,1})$. 
Now notice that the orbifold $S (A_7, \CC^\times)$ is the moduli space of $8$-element subsets of $\CC^\times$ given up a common scalar 
and up to a (common) inversion. This is also the moduli space of 10-element subsets of a projective line endowed with a distinguished subset of 2 elements. If we pass to double covers ramifying over the remaining 8 points, we see that this is nothing but  $\PP(\Hcal (1^4)_\hyp)$. With this interpretation, the added locus is even identified with   $\PP(\Hcal (1^4)_\hyp)$ as an orbifold. 

In order to make the construction global it is best to pass to a level two structure: consider the universal elliptic curve with a level two structure
$\Jcal^*[2]/\Mcal_{1,1}[2]$ (a Deligne-Mumford stack). It comes with a $\SL (2, \ZZ/2)$-action. We extend this to across the completed modular curve $\overline{\Jcal}[2]/\overline{\Mcal}_{1,1}[2]$ with curve of type $I_2$ added as singular fibers. Denote by $\Jcal [2]/\overline{\Mcal}_{1,1}[2]$  the associated abelian stack and define  the stack $\Hom_{(A_7)}(Q(\root),\Jcal [2]/\overline{\Mcal}_{1,1}[2])$ in an evident manner. We then put 
\[
S_{(A_7)}(E_7,\overline{\Ccal}_{1,1}/\overline{\Mcal}_{1,1}):= \aut (\root)\bs\Hom_{(A_7)}(Q(\root),
\Jcal [2]/\overline{\Mcal}_{1,1}[2])^\circ/
\SL (2, \ZZ/2).
\]
The right hand side contains $S(E_7,\overline{\Ccal}_{1,1})/\overline{\Mcal}_{0,(4)}$, where the point added to
$\overline{\Mcal}_{0,(4)}$ (the cusp) is represented by a divisor on $\PP^1$ which is twice a positive reduced divisor of degree two. Its complement is a  $\ZZ/2$-gerbe over $S(A_7,\CC^\times)$ and we conclude:

\begin{theorem}\label{thm:1,1,1,1}
We have a natural isomorphism $\PP\Hcal (1^4)\cong S_{(A_7)}(E_7,\overline{\Ccal}_{1,1})$ over
$\overline{\Mcal}_{0,(4)}$ which identifies $\PP\Hcal (1^4)_\hyp$ with the fiber over the cusp (which, as we noted, has the structure of a  $\ZZ/2$-gerbe over $S(A_7,\CC^\times)$).
\end{theorem}
This theorem gives us, at least in principle,  access to the homotopy type of $\PP\Hcal (1^4)$, although we admit that 
this may be hard in practice. A computation of its orbifold fundamental group looks feasible, however.

\section{The remaining strata} 

\subsection*{The stratum $\PP\Hcal(2,1^2)$.}
We return to the limiting discussion in Section \ref{sect:bianti}.  But now we  assume  here that $L_t$ is tangent to $C_t$
at one point and $L_t$
limits to $L_0=\R+\R_{b_0}$, where $\R$ is an ordinary ray and
$\R_{b_0}$ is a Weierstra\ss\ ray.
The construction of the previous section now produces a family $\Lcaltil/\Delta$
for which $\Lcaltil^*/\Delta^*$ is a nodal curve of genus $1$, whose closed fiber $\Ltil_0$  has three components:
$R_{b_0}^\pm$ and $\Rtil$.

The exceptional curves meeting $\Ltil_0$
at the singular point $v^\pm$
are $\Ecal^{\pm}_{b_i,b_j}$, $1\le i<j\le 7$. The roots that are differences
of two disjoint members taken from this collection make up a root subsystem $\root_0\subset \root$ of type $A_6$ having  $(e_1-e_2,\dots ,e_6-e_7)$ as a root basis. 
Moreover, the analogous of Lemma~\ref{lemma:specialize} holds,
namely: if $\alpha\in \root$, then $\chi_{\Ccal/\Dcal}(\alpha)$ specializes to an element of $\pic^0(\Ltil_0)\cong\CC^\times$
if and only if $\alpha\in \root_0$. 

%

This suggests the following construction (taken from \cite{looij2008}). 
Consider the torus $\Hom (Q(\root),\CC^\times)$. Its lattice of one parameter subgroups can be identified with the weight lattice 
$\Hom (Q(\root),\ZZ)$ and hence its tensor product  with $\RR$ with $\Hom (Q(\root),\RR)$. The indivisible elements in $\Hom (Q(\root),\ZZ)$ whose kernel is 
root lattice of a subsystem of type $A_6$ make up a $W(\root)$-orbit $\Ofrak$ of a fundamental weight. Each of these elements spans an oriented ray
in
$\Hom (Q(\root),\RR)$ and the collection of such rays defines a toric extension 
\[
\Hom (Q(\root),\CC^\times)\subset \Hom_{(A_6)}(Q(\root),\CC^\times).
\] 
To every subsystem $\root_0\subset \root$ of type $A_6$ are associated two $\RR_+$-rays and hence two copies of $\Hom (Q(\root_0),\CC^\times)$. So if $\we\in\Ofrak$ spans of one of the rays and if we let $\CC^\times$ act on $\PP^1$ in the usual manner: $\zeta [z_0:z_1] =[\zeta z_0:z_1]$, then we can form $\PP^1\times_\we\Hom (Q(\root),\CC^\times)$ (which is isomorphic to $\PP^1\times (\CC^\times)^6$) and this glues the two copies of
$\Hom (Q(\root_0),\CC^\times)$ on $\Hom (Q(\root),\CC^\times)$. Notice that $W(\root)$ acts on $\Hom_{(A_6)}(Q(\root),\CC^\times)$. The $W(\root)$-stabilizer of the boundary torus defined  by $\we\in\Ofrak$ is the Weyl group of the $A_6$-subsystem defined by $\we$.

Let $\Hom_{(A_6)}(Q(\root),\CC^\times)^\circ$ be obtained by removing from $\Hom_{(A_6)}(Q(\root),\CC^\times)$ the fixed point loci of the reflections in $W(\root)$ and put
\[
S_{(A_6)}(E_7,\CC^\times):= \aut (E_7)\bs\Hom_{(A_6)}(Q(\root), \CC^\times)^\circ/\aut(\CC^\times).
\]
This contains $S(E_7,\CC^\times)$ as an open subset. Since the $\aut (E_7)=W(E_7)$ and the $W(E_7)$-stabilizer of a toric stratum in $\Hom_{(A_6)}(Q(\root), \CC^\times)$ is a Weyl group of type $A_6$, the added locus is isomorphic to $W(A_6)\bs\Hom (Q(A_6), \CC^\times)^\circ$.
Argueing as before (see also \cite{looij2008}), we find

\begin{theorem}\label{thm:2,1,1}
We have a natural isomorphism of  orbifolds  
\[
\PP\Hcal (2,1^2)\cong S_{(A_6)}(E_7,\CC^\times)
\]
which extends the $\ZZ/2$-gerbe $\PP\Hcal (2,1^2)_\hyp\to W(A_6)\bs \Hom (Q(A_6),\CC^\times)^\circ$
of Proposition \ref{prop:hypergerbe}.
\end{theorem}

\subsection*{The stratum $\PP\Hcal(2,2)$}
Here we need to deal with the case when $L_t$ tangent to
$C_t$ in two points and $L_t$ 
limits to the union $L_0=\R_{b_0}+\R_{b_7}$ of two Weierstra\ss\ rays.
Then $\Lcaltil/\Delta$ is such that  $\Lcaltil^*/\Delta^*$ is s bigon curve and the closed fiber $\Ltil_0$  has four irreducible
components: $\R_{b_0}^\pm$ and $\R_{b_7}^\pm$.
If $\Ycal^*/\Delta^*$ is given by blowing up 7 numbered points $p_1,\dots ,p_7$ in $\PP^2_{\Delta^*}$ as before, then
%
the root system $\root':=\root(\Ycal^*/\Delta^*,\Lcaltil^*/\Delta^*)$ is of type $E_6$ and has root basis $(\ell-e_1-e_2-e_3,e_1-e_2,\dots ,e_5-e_6)$. Via  the identification described in Section \ref{sect:bianti} we find that the exceptional curves through $v^\pm$ and without common components
with $\Ltil_0$ are
$\Ecal^\pm_{b_i,b_j}$ with $1\le i<j\le 6$. 
The roots that are differences of two disjoint members taken from this collection,
and meeting both $v^+$ or both $v^-$,
make up a root subsystem $\root'_0\subset \root'$ of type $A_5$ having  $(e_1-e_2,\dots ,e_5-e_6)$ as root basis. If $\alpha\in \root'$, then $\chi_{\Ccal/\Dcal}(\alpha)$ specializes to an element of $\CC^\times$ if and only if $\alpha\in \root'_0$.  

A construction similar to the one for the $\PP \Hcal (2,1^2)$-stratum  then yields:
\begin{theorem}\label{thm:2,2}
We have a natural isomorphism of  orbifolds  
\[
\PP\Hcal (2,2)\cong S_{(A_5)}(E_6,\CC^\times)
\]
which extends the  $\ZZ/2$-gerbe 
$\PP\Hcal (2,2)_\hyp\to S(E_6,\CC^\times)$ of Proposition \ref{prop:hypergerbe}.
\end{theorem}

\subsection*{Orbifold fundamental groups}
The orbifold fundamental groups of an orbifold of the type  $S_{(\root_0)}(\root,\CC^\times)$, where $\root$ is an irreducible and reduced  root system and $\root_0\subset \root$ is a saturated root subsystem of corank one has essentially been determined in \cite{looij2008}.  It is best described 
in terms of the extended (affine) root system, or rather, of the associated affine Coxeter  system. We briefly recall the construction.
Although much of what follows holds in greater generality, let us confine ourselves here to the case when $\G$ is an affine Coxeter diagram of type $\hat E_7$ (resp.\ $\hat E_6$): this is $T$-shaped tree whose arms have edge length $3,3,2$ (resp.\ $2,2,2$).
So the automorphism group $\aut (\G)$ of $\G$ is a permutation group of a set of 2
(resp.\  3 elements). Denote its vertex set by $I$. Then we have defined an associated Artin group $\art_\G$ given in terms of generators and relations: 
the generators are indexed by $I$: $\{t_i\}_{i\in I}$ and $t_i$ commutes with $t_j$ unless $i$ and $j$ span an edge of $\G$ in which case we have a braid relation $t_it_jt_i=t_jt_it_j$ in the following manner. The group $\aut (\G)$ acts on $\art_\G$ by permuting its generators.

To $\G$ is associated a Coxeter group $W_\G$ (the quotient of $\art_\G$ by putting $t_i^2\equiv 1$ for all $i\in I$)  and a (Tits) representation of the Coxeter group on a real affine
space $\Aff$ on which $W_\G$ acts properly as an affine reflection group. The generating set $I$ defines a fundamental simplex $\simpl\subset \Aff$. The group $\aut (\G)$ acts as a symmetry group on $\simpl$ and this action extends affine-linearly to $\Aff$. Thus $W_\G\rtimes \aut(\G)$ (a quotient  of $\art_\G\rtimes \aut(\G)$) acts on $\Aff$. Let $\Aff_\CC^\circ$ denote the complexification of $\Aff$ with all its (affine) reflection hyperplanes removed. 
Then $\Aff_\CC^\circ$ can be identified with a $W_\G\rtimes \aut(\G)$-covering of $S(\root,\CC^\times)$ and $\art_\G\rtimes \aut(\G)$ can be identified with the orbifold fundamental group of $S(\root,\CC^\times)$ in such a manner that the covering projection
$\Aff_\CC^\circ\to S(\root,\CC^\times)$ is given by the natural map $\art_\G\rtimes \aut(\G)\to W_\G\rtimes \aut(\G)$ (\cite{looij2008}, Cor.\ 3.7).
 
The inclusion $S(\root,\CC^\times)\subset S_{(\root_0)}(\root,\CC^\times)$ induces a surjection on fundamental groups and essentially amounts to introducing one new relation: a loop around the added divisor gets killed in the fundamental group. The question is therefore how to represent that loop in $\art_\G\rtimes \aut (\G)$. This was addressed in \cite{looij2008} (Lemma 3.8 ff.). Let us describe this in some detail.

For every $i\in I$, the  subgraph $\G_i\subset \G$ obtained by removing $i$ and the edges connected to it is the graph of a finite Coxeter group  $W_{\G_i}$ which maps isomorphically onto the $W_\G$-stabilizer of a vertex $v_i$ of $\simpl$. The homomorphism $\art_{\G_i}\to \art_{\G}$ is known to be an embedding. We denote by $\Delta_{i}$ the image of the \emph{Garside element} (see \emph{op.\ cit.}) of $\art_{\G_i}$. Its image in
$W_{\G_i}$ is the longest element $w_i$ of $W_{\G_i}$ and $w_i$ takes  $\simpl$  to a simplex opposite $v_i$. The opposition symmetry $\refl_i : \Aff\ni a\mapsto v_i -(a-v_i)\in \Aff$ with respect to $v_i$ composed with $w_i$ is represented by an automorphism of $\G_i$ and $\refl_i w_i$ preserves  $\simpl$ if and only if this automorphism is the restriction of some  $g_i\in\aut (\G)$. Let us call $i\in I$ \emph{quasi-special} if that is the case.
Then for a quasi-special $i\in I$ we have $g_i\Delta_{i}=\Delta_{i}g_i$ in $\art_\G\rtimes \aut (\G)$ and this element acts on $\Aff$ as $\refl_i$.
If we have two distinct  vertices $i,j\in I$ that are quasi-special, then $\refl_j\refl_i$ acts in $\Aff$ as translation over $2(v_j-v_i)$.
Now $2(v_j-v_i)$ defines a one parameter subgroup  $\subgr: \CC^\times\to \Hom (Q,\CC^\times)$ and under the identification of $\art_\G\rtimes \aut (\G)$ with the orbifold fundamental group of $S(\root,\CC^\times)$, the lift  $(\Delta_{j}g_j)(\Delta_{i}g_i)^{-1}$ of $\refl_j\refl_i$ represents the conjugacy class of a simple loop in $S(\root,\CC^\times)$. This is the loop can be obtained by taking 
the $\subgr$-image of  a circle $|z|=\varepsilon$ of small radius in $\Hom (Q,\CC^\times)$, applying a translate under the torus  $\Hom (Q,\CC^\times)$ so that the circle lies in $\Hom (Q,\CC^\times)^\circ$ and then mapping that circle to $S(\root,\CC^\times)$. 
We have $(\Delta_{j}g_j)(\Delta_{i}g_i)^{-1}=\Delta_{j}g_jg_i^{-1}\Delta_{i}^{-1}$ and so,
if we have a toric  extension in which such loops become contractible, then in this extension we have the relation $\Delta_{j}^{-1}\Delta_{i}=g_jg_i^{-1}$. 
Let us now treat the  two cases separately. 
\\

Assume $\G$ of type  $\hat E_7$. We take as quasi-special vertices one for which $\G_i$ is of type $E_7$ (then the 
associated element of $\aut(\G)$ is the identity) and the unique one for which $\G_j$ is of type $A_7$ (then the 
associated  $g\in\aut(\G)$ is not the identity). The loop in question is associated to the translation $v_j-v_i$ and so we are imposing the identity $\Delta_{\G(A_7)}^{-1}\Delta_{\G(E_7)}=g$. Hence we can eliminate $g$ and we find:

\begin{theorem}\label{thm:fund-group-2,1,1}
The orbifold fundamental group of $\PP\Hcal(2, 1^2)$ is the largest quotient of the affine Artin group $\art_{\hat E_7}$ for which 
$\Delta_{\G(A_7)}^{-1}\Delta_{\G(E_7)}$ is of order $2$ and conjugation by $\Delta_{\G(A_7)}^{-1}\Delta_{\G(E_7)}$ permutes the generators 
(indexed by $I$) according to the nontrivial involution. 
\end{theorem}

Assume  now $\G$ of type $\hat E_6$. We let $i\in I$ be a terminal vertex (so that $\G_i$ is of type $E_6$)  and let $j\in I$ be the unique vertex $\not=i$ connected with $i$ (so that $\G_j$ is of type $A_5+A_1$). Both are quasi-special, define subgraphs $\G(E_6)$ and $\G(A_5+A_1)$ (here the notation indicates the type) and define the same element  of $\aut (\G)$ (namely the unique involution which fixes $i$). Then the loop in question is represented by $\Delta_{\G(A_5+A_1)}^{-1}\Delta_{\G(E_6)}$ and hence:

\begin{theorem}\label{thm:fund-group-2,2}
The orbifold fundamental group of $\PP\Hcal(2, 2)$ is the quotient of the semidirect product of the affine Artin group $\art_{\hat E_6}$ 
with the symmetry group $\aut (\G_{\hat E_6})$ of the  $\hat E_6$-graph defined by the relation $\Delta_{\G(A_5+A_1)}\equiv\Delta_{\G(E_6)}$. 
\end{theorem}

We may also write this as a semidirect product of $\aut (\G_{\hat E_6})$ and a quotient of $\art_{\hat E_6}$. The quotient is then obtained by imposing three such relations: one for every terminal vertex, so that $\aut (\G_{\hat E_6})$ still acts on it.

Needless to say that we don't know whether any of these has a contractible orbifold universal cover.

\end{document}